\documentclass[twoside]{article}
\usepackage{amsmath,amsthm,amssymb,amsfonts,amscd,array}
\usepackage[all]{xy}
\usepackage{mathrsfs}

\title{$n$-$\mathscr{X}$-Coherent Rings}
\date{}

\newtheorem{thm}{Theorem}[section]
 
 \newtheorem{lem}[thm]{Lemma}
 
 \newtheorem{defn}[thm]{Definition}

 \newtheorem{exs}[thm]{Examples}
\catcode`\ç=13
\defç{\c{c}}
\catcode`\é=13
\defé{\'e}
\catcode`\à=13
\defà{\`a}
\catcode`\è=13
\defè{\`e}
\catcode`\â=13
\defâ{\^a}
\catcode`\ù=13
\defù{\`u}
\catcode`\ê=13
\defê{\^e}
\catcode`\î=13
\defî{\^\i}
\catcode`\ô=13
\defô{\^o}
\newcommand{\field}[1]{\mathbb{#1}}

\newcommand{\Q }{\field{Q}}
\newcommand{\Z }{\field{Z}}

\def\Card{{\rm Card}}

\def\Ima{{\rm Im}}

\def\Ext{{\rm Ext}}
\def\Tor{{\rm Tor}}
\def\Hom{{\rm Hom}}

\begin{document}
\thispagestyle{empty}
%%%%%%%%%%%%%%%%%%%%%%%%%%%%%%%%%%%%%%%%%%%%%%%%%%%%%%%%%
%%%%%%%%%%%%%%%%%%%%%%%%%%%%%%%%%%%%%%%%%%%%%%%%%%%%%%%%%

%%%%%%%%%%%%%%%%%%%%%%%%%%%%%%%%%%%%%%%%%%%%%%%%%%%%%%%%%
%%%TITLE%%%%%%%%%%%%%%%%%%%%%%%%%%%%%%%%%%%%%%%%%%%%%%%%%
\maketitle \vspace*{-1.5cm}

%%%%%%%%%%%%%%%%%%%%%%%%%%%%%%%%%%%%%%%%%%%%%%%%%%%%%%%%%
%%%%%%%%%%%%%%%%%%%%%%%%%%%%%%%%%%%%%%%%%%%%%%%%%%%%%%%%%
%%%%%%%%%%%%%%%%%%%%%%%%%%%%%%%%%%%%%%%%%%%%%%%%%%%%%%%%%
%%%NAMES%%%%%%%%%%%%%%%%%%%%%%%%%%%%%%%%%%%%%%%%%%%%%%%%%
\begin{center}
{\large\bf Driss Bennis}

\bigskip
%%%%%%%%%%%%%%%%%%%%%%%%%%%%%%%%%%%%%%%%%%%%%%%%%%%%%%%%%
%%%%%%%%%%%%%%%%%%%%%%%%%%%%%%%%%%%%%%%%%%%%%%%%%%%%%%%%%
%%%%%%%%%%%%%%%%%%%%%%%%%%%%%%%%%%%%%%%%%%%%%%%%%%%%%%%%%
%%%ADDRESSES%%%%%%%%%%%%%%%%%%%%%%%%%%%%%%%%%%%%%%%%%%%%%

%%%%%%%%%%%%%%%%%%%%%%%%%%%%%%%%%%%%%%%%%%%%%%%%%%%%%%%%%
%%%%%%%%%%%%%%%%%%%%%%%%%%%%%%%%%%%%%%%%%%%%%%%%%%%%%%%%%
%%%%%%%%%%%%ADDRESSES%%%%%%%%%%%%%%%%%%%%%%%%%%%%%%%%%%%%%%%%%%%%%
Department of Mathematics, Faculty of Science and Technology of
Fez,\\ Box 2202, University S. M. Ben Abdellah Fez,
Morocco, \\[0.2cm]
 driss\_bennis@hotmail.com
\end{center}

\bigskip\bigskip
%%%%%%%%%%%%%%%%%%%%%%%%%%%%%%%%%%%%%%%%%%%%%%%%%%%%%%%%%
%%%%%%%%%%%%%%%%%%%%%%%%%%%%%%%%%%%%%%%%%%%%%%%%%%%%%%%%%
%%%%%%%%%%%%%%%%%%%%%%%%%%%%%%%%%%%%%%%%%%%%%%%%%%%%%%%%%
%%%ABSTRACT%%%%%%%%%%%%%%%%%%%%%%%%%%%%%%%%%%%%%%%%%%%%%%
\noindent{\large\bf Abstract.} This paper unifies several
generalizations of coherent rings in one notion. Namely, we
introduce $n$-$\mathscr{X}$-coherent rings, where $\mathscr{X}$ is
a class of modules and $n$ is a positive integer, as those rings
for which the subclass $\mathscr{X}_n$ of $n$-presented
 modules of $\mathscr{X}$ is not empty, and every module in
$\mathscr{X}_n$ is $n+1$-presented. Then, for each particular
class $\mathscr{X}$ of modules, we find correspondent relative
coherent rings.\\
Our main aim is to show that the  well-known Chase's, Cheatham and
Stone's, Enochs', and Stenstr\"{o}m's characterizations of
coherent rings hold true for any $n$-$\mathscr{X}$-coherent
rings.\bigskip

%%%%%%%%%%%%%%%%%%%%%%%%%%%%%%%%%%%%%%%%%%%%%%%%%%%%%%%%%
\small{\noindent{\bf Key Words.}   $n$-$\mathscr{X}$-coherent
rings, $n$-$\mathscr{X}$-flat modules, $n$-$\mathscr{X}$-injective
 modules, Charter modules, Preenvelopes}\medskip

\small{\noindent{\bf 2000 Mathematics Subject Classification.}
16D80, 16E05, 16E30, 16E65, 16P70}
\bigskip\bigskip

%%%%%%%%%%%%%%%%%%%%%%%%%%%%%%%%%%%%%%%%%%%%%%%%%%%%%%%%%
%%%%%%%%%%%%%%%%%%%%%%%%%%%%%%%%%%%%%%%%%%%%%%%%%%%%%%%%%
%%%INTRODUCTION%%%%%%%%%%%%%%%%%%%%%%%% %%%%%%%%%%%%%%%%%%

\section{Introduction }   Throughout this paper, $R$ denotes
a non-trivial associative ring with identity,  and all $R$-modules
are, if not specified otherwise, left $R$-modules. For an
$R$-module $M$, we use $M^*$   to denote the character module $
\Hom_{\Z}(M, \Q/ \Z)$ of $M$. An $R$-module $M$ is said to be
$n$-presented, for a positive integer $n$, if there is an exact
sequence  of $R$-modules:
 $F_n\rightarrow F_{n-1}\rightarrow \cdots\rightarrow F_0\rightarrow
M\rightarrow 0$    where each $F_i$ is finitely generated and
free. In particular, $0$-presented  and $1$-presented modules are
finitely generated and finitely presented  modules respectively.
An $R$-module $M$ is said to be infinitely presented, if it is
$m$-presented for every positive integer $m$.\bigskip

A ring $R$ is called left coherent, if every finitely generated
left ideal is finitely presented, equivalently every finitely
presented $R$-module is $2$-presented and so infinitely presented.
The coherent rings were first appear in Chase's paper \cite{Chase}
without being mentioned by name. The term coherent was first used
by Bourbaki in \cite{Bou2}. Since then, coherent rings have became
a vigorously active area of research (see Glaz's book \cite{Glaz}
for more details).\bigskip

Several characterizations of coherent rings have been done by
various notions. Here, we are interested in the following
homological ones:
\begin{itemize}
    \item In \cite{Chase}, Chase characterized left
coherent rings as those rings over which every direct product of
flat right modules is flat.
    \item In \cite{Stenst},  Stenstr\"{o}m
proved that a ring $R$ is left coherent if and only if every
direct limit of FP-injective $R$-modules is also FP-injective.
    \item Cheatham and Stone \cite[Theorem 1]{CS} showed that  coherent
rings can be characterized by the use of the notion of character
module as follows:\\
The following assertions are equivalent:
\begin{enumerate}
    \item   $R$ is left coherent;
     \item  An $R$-module $M$ is   injective if and
     only if $M^*$ is flat;
 \item   An $R$-module  $M$ is   injective if and
     only if $M^{**}$ is  injective;
    \item A right $R$-module  $M$ is   flat if and
     only if $M^{**}$ is  flat.
     \end{enumerate}
    \item The notion of flat preenvelopes of modules is used by Enochs to
characterize coherent rings. Recall  that  an $R$-module $F$ in
some class of $R$-modules $\mathscr{X}$ is said to be an
$\mathscr{X}$-preenvelope of an $R$-module $M$, if there is a
homomorphism   $\varphi :\, M\rightarrow F$ such that, for any
homomorphism $\varphi' :\, M\rightarrow F'$ with $F'\in
\mathscr{X}$, there is a homomorphism $f :\, F\rightarrow F'$ such
that $\varphi' =f\varphi$ (see \cite{Rel-hom} for more details
about this notion). The homomorphism $\varphi$ is also called an
$\mathscr{X}$-preenvelope of  $M$. If $\mathscr{X}$ is the class
of flat $R$-modules, an $\mathscr{X}$-preenvelope of $M$ is simply
called a flat preenvelope of $M$. We have \cite[Proposition
6.5.1]{Rel-hom}: a ring $R$ is left coherent if and only if every
right $R$-module has a flat preenvelope.
\end{itemize}

The above characterizations of coherent rings have led to
introduce various  relative coherent rings (see Examples
\ref{ex-1} for some of these  rings). Namely, for each relative
coherent rings, relative flat and injective modules were
introduced and so  used to give characterizations  of their
corespondent  relative coherent rings in the same way as Chase's,
Cheatham and Stone's, Enochs', and Stenstr\"{o}m's
characterizations of coherent rings cited above (see references
for more details). The idea of this paper is to unify these
relative coherent rings in one notion which we call
$n$-$\mathscr{X}$-coherent rings, where $\mathscr{X}$ is a class
of modules and $n$ is a positive integer (see Definition
\ref{def-Coh-ring}). As main results of this paper, we give a
generalization of the above characterizations of coherent rings to
the setting of  $n$-$\mathscr{X}$-coherent rings (see Theorems
\ref{thm-flatr-inj}, \ref{thm-character}, \ref{thm-prenv}). So,
relative flat and injective modules are introduced (see Definition
\ref{def-flat-injective}). Before giving the desired results, we
begin with a characterization of $n$-$\mathscr{X}$-coherent rings
in terms of the functors $\Tor$ and $\Ext$ (see Theorem
\ref{thm-TOr-Ext}).

%%%%%%%%%%%%%%%%%%%%%%%%%%%%%%%%%%%%%%%%%%%%%%%%%%%%%%%%%
%%%%%%%%%%%%%%%%%%%%%%%%%%%%%%%%%%%%%%%%%%%%%%%%%%%%%%%%%
%%% Section 2.   %%%%%%%%%%%%%%%%%%%%%%%%%%%%%%%%%%%%%%%%%%

\section{Main results}

In this paper we are concerned with the following generalization
of the notion of coherent rings.

\begin{defn}\label{def-Coh-ring}  Let
$\mathscr{X}$ be a class of $R$-modules.
\begin{itemize}
    \item $R$ is said to be left $n$-$\mathscr{X}$-coherent, for a positive
integer $n$, if the subclass $\mathscr{X}_n$ of $n$-presented
$R$-modules of $\mathscr{X}$ is not empty, and every $R$-module in
$\mathscr{X}_n$ is $n+1$-presented.
    \item Similarly, the right
$n$-$\mathscr{X}$-coherent rings are defined.
    \item A ring $R$ is called   $n$-$\mathscr{X}$-coherent if it is
both left and  right $n$-$\mathscr{X}$-coherent.
\end{itemize}
\end{defn}

It is trivial to show that over $n$-$\mathscr{X}$-coherent rings
the  $n$-presented modules are in fact infinitely presented.

\begin{exs}\label{ex-1}
\begin{enumerate}
    \item Clearly, for $n=0$ and $\mathscr{C}$ is  the class of all cyclic
$R$-modules, the $1$-$\mathscr{C}$-coherent rings are just the
 Noetherian rings.
    \item  For $n=1$, the $1$-$\mathscr{C}$-coherent rings are
    just the
  coherent rings. Note that, from \cite[Theorem
2.3.2]{Glaz}, the $1$-$\mathscr{C}$-coherence is   the same as the
$1$-$\mathscr{M}$-coherence, where $\mathscr{M}$ denotes the class
of all $R$-modules.
\item  An extension of the notion of coherent rings were
introduced in \cite{Costa} and \cite{DKM} as follows: for any
positive integer $n\geq 1$,  a ring $R$ is called $n$-coherent
(resp., strong $n$-coherent), if it is $n$-$\mathscr{C}$-coherent
(resp., $n$-$\mathscr{M}$-coherent). The  strong $n$-coherent
rings were introduced by Costa \cite{Costa} who first called them
$n$-coherent (see also \cite{CD, DKMS, KM}).
    \item Let $s$ and $t$ be two positive integer and let
$\mathscr{M}_{(s,t)}$ be the class
 of finitely presented $R$-modules  of the form
$R^s/K$, where $K$ is a $t$-generated submodule of  the left
$R$-module $R^s$. The $1$-$\mathscr{M}_{(s,t)}$-coherent rings
were introduced in \cite{CZZ} and they were called
$(s,t)$-coherent rings. In particular, The
$1$-$\mathscr{M}_{(1,1)}$-coherent rings (equivalently, the rings
that satisfy: every principal ideal is finitely presented) were
introduced in \cite{DM3} and they were called P-coherent rings.
    \item Also a left min-coherent ring were introduced in  \cite{Mao3} as a
particular case of $1$-$\mathscr{M}_(s,t)$-coherent rings: a ring
$R$ is said to be   left min-coherent if every simple left ideal
of $R$ is finitely presented. Then min-coherent rings are just the
$1$-$\mathscr{C}_{S}$-coherent, where $\mathscr{C}_{S}$ is the
class of all cyclic $R$-modules of the form $R/I$, where $I$ is a
simple left ideal  of $R$.
    \item In the case where $\mathscr{X}$ is the class of all
submodules of the Jacobson radical,   the
$0$-$\mathscr{X}$-coherent rings are called J-coherent (see
\cite{DLM}).
 \item In \cite{Camillo} (see also \cite{Mao2}),    the class
 $\mathscr{T}$  of all
 torsionless $R$-module is of interest, such that the $0$-$\mathscr{T}$-coherent rings
are called  $\Pi$-coherent ring.
 \item Finally, consider a class $\mathscr{P}_d$ of all modules of
 projective dimension at most a positive integer $d$. The $m$-$\mathscr{P}_d$-coherent rings were introduced in \cite{DM2} and
 they were called $(m,d)$-coherent (this notion differs from the one in 4 above). The $1$-$\mathscr{P}_d$-coherent rings were first introduced  in
 \cite{Lee} and they were called $d$-coherent  (this notion also differs from the one in 3 above). Also, a particular
 case of $1$-$\mathscr{P}_d$-coherent were introduced in
 \cite{LWei}.
\end{enumerate}
\end{exs}

All of the above relative coherent rings have analogous
characterizations of Chase's, Cheatham and Stone's, Enochs', and
Stenstr\"{o}m's characterizations  of coherent rings (see
references). The aim of this paper is to show that all of these
characterizations hold true for $n$-$\mathscr{X}$-coherent rings
without any further condition on the class of modules
$\mathscr{X}$.\bigskip

We begin with the following characterization of
$n$-$\mathscr{X}$-coherent rings in terms of the functors $\Tor$
and $\Ext$

\begin{thm}\label{thm-TOr-Ext} Let
$\mathscr{X}$ be a class of $R$-modules such that, for a positive
integer $n\geq 1$, the subclass $\mathscr{X}_n$ of $n$-presented
$R$-modules of $\mathscr{X}$ is not empty.  Then, the following
assertions are equivalent:
\begin{enumerate}
    \item   $R$ is left $n$-$\mathscr{X}$-coherent;
     \item  For every set $J$, the canonical
homomorphism  $ R^J\otimes_R M\rightarrow M^J $ is bijective for
$M\in \mathscr{X}_n$ , and we have $\Tor^R_i( R^J,M)= 0$
 for every $0<i\leq n$;
 \item For every family $(P_{j})_{j\in J}$ of right $R$-modules, the canonical
homomorphism:  $$ \displaystyle\prod_{j\in J}
\Tor^R_i(P_{j},M)\rightarrow
 \Tor^R_i(\displaystyle\prod_{j\in J} P_{j}, M)$$ is bijective for every $i\leq n$  and every $M\in
 \mathscr{X}_n$;
    \item For every direct system $(N_{j})_{j\in J}$
of  $R$-modules over a directed index set $J$, the canonical
homomorphism:  $$ \underrightarrow{lim}
\Ext^i_R(M,N_{j})\rightarrow
 \Ext^i_R(M,\underrightarrow{lim}N_{j})$$ is bijective for every $i\leq n$ and every $M\in \mathscr{X}_n$.
     \end{enumerate}
\end{thm}
\begin{proof} All equivalences follow from the following result.
\end{proof}

\begin{lem}[\cite{Bou}, Exercice 3,
page187]\label{lem-TOr-Ext}  Let  $M$ be an $R$-module. For a
positive integer $n\geq 1$, the following assertions are
equivalent:
\begin{enumerate}
    \item $M$ is $n$-presented;
     \item  For every set $J$, the canonical
homomorphism  $ R^J\otimes_R M\rightarrow M^J $ is bijective, and
we have $\Tor^R_i( R^J,M)= 0$  for every $0<i< n$;
 \item For every family $(P_{j})_{j\in J}$ of right $R$-modules, the canonical
homomorphism:  $$ \displaystyle\prod_{j\in J}
\Tor^R_i(P_{j},M)\rightarrow
 \Tor^R_i(\displaystyle\prod_{j\in J} P_{j}, M)$$ is bijective for every $i<
 n$;
    \item For every direct system $(N_{j})_{j\in J}$
of  $R$-modules over a directed index set $J$, the canonical
homomorphism:  $$ \underrightarrow{lim}
\Ext^i_R(M,N_{j})\rightarrow
 \Ext^i_R(M,\underrightarrow{lim}N_{j})$$ is bijective for every $i< n$.
     \end{enumerate}
\end{lem}

We also use the above result to characterize
$n$-$\mathscr{X}$-coherent rings by relative flatness and
injectivity, which are defined as follows:

\begin{defn}\label{def-flat-injective}  Let
$\mathscr{X}$ be a class of $R$-modules  such that, for a positive
integer $n\geq 1$, the subclass $\mathscr{X}_n$ of $n$-presented
$R$-modules of $\mathscr{X}$ is not empty.
\begin{itemize}
    \item A right $R$-module $M$ is called $n$-$\mathscr{X}$-flat if  $\Tor^R_n(M,N)=
    0$ for every $N\in \mathscr{X}_n$. The  $n$-$\mathscr{X}$-flat left $R$-modules are
    defined similarly.
    \item An $R$-module $M$ is called $n$-$\mathscr{X}$-injective if  $\Ext_R^n(N,M)=
    0$ for every $N\in \mathscr{X}_n$.
\end{itemize}
\end{defn}

 As in Example \ref{ex-1},  we get,  for each special class $\mathscr{X}$ of modules, a correspondent relative flatness
and injectivity. For instance, if $\mathscr{C}$ is the class of
all cyclic $R$-modules, the $1$-$\mathscr{C}$-flat right
$R$-modules are just the classical flat  right $R$-module, and the
$1$-$\mathscr{C}$-injective $R$-modules are just the FP-injective
$R$-modules \cite{Stenst}. If $\mathscr{M}$ is  the class of all
$R$-modules, then,  for a positive integer $n\geq 1$,  the
$n$-$\mathscr{M}$-flat right  $R$-modules were called in \cite{CD}
$n$-flat right $R$-module, and the $n$-$\mathscr{M}$-injective
$R$-modules were called $n$-injective modules.\bigskip

Now we give our first main result, which is a generalization of
Chase's  and Stenstr\"{o}m's characterizations  of coherent rings.

\begin{thm}\label{thm-flatr-inj} Let
$\mathscr{X}$ be a class of $R$-modules such that, for a positive
integer $n\geq 1$, the subclass $\mathscr{X}_n$ of $n$-presented
$R$-modules of $\mathscr{X}$ is not empty.  Then, the following
assertions are equivalent:
\begin{enumerate}
    \item   $R$ is left $n$-$\mathscr{X}$-coherent;
     \item  For every set $J$, the right  $R$-module $ R^J$ is
     $n$-$\mathscr{X}$-flat;
 \item For every family $(P_{j})_{j\in J}$ of  $n$-$\mathscr{X}$-flat right $R$-modules,  the direct product $
  \displaystyle\prod_{j\in J} P_{j} $ is  $n$-$\mathscr{X}$-flat;
    \item For every direct system $(N_{j})_{j\in J}$
of $n$-$\mathscr{X}$-injective   $R$-modules over a directed index
set $J$, the direct limit $\underrightarrow{lim}N_{j}$ is
$n$-$\mathscr{X}$-injective.
     \end{enumerate}
\end{thm}
\begin{proof} The implication $1\Rightarrow 3$   follows from Theorem
\ref{thm-TOr-Ext} ($1\Leftrightarrow 3$). The implication $2\Rightarrow 3$ is obvious.\\
We prove the implication
 $2\Rightarrow 1$. Consider an
$R$-module $M\in \mathscr{X}_n$. Then, there is an exact sequence
of $R$-modules: $$ F_{ n }\rightarrow F_{ n-1}\rightarrow \cdots
\rightarrow F_{ 1} \rightarrow F_0\rightarrow M \rightarrow 0$$
such that each $F_i$ is  finitely generated and free. Consider $
K_{ n }= \Ima( F_{ n }\rightarrow F_{ n-1})$ and $ K_{ n-1 }=
\Ima( F_{ n-1 }\rightarrow F_{ n-2})$. Then, we have the following
short exact sequence $$ 0 \rightarrow  K_{ n } \rightarrow F_{
n-1}\rightarrow K_{ n-1}\rightarrow 0 $$ Since
$\Tor^R_1(N,K_{n-1})\cong\Tor^R_n(N,M)
    =0$ for every $N\in \mathscr{X}_n$, we get the following
    commutative diagram with exact rows:
    $$\xymatrix{
    0 \ar[r]   &   R^J\otimes_R  K_{ n } \ar[d]^{\alpha} \ar[r]   &  R^J \otimes_R F_{ n-1 } \ar[d]^{\beta} \ar[r]  & R^J \otimes_R  K_{ n-1 } \ar[d]^{\gamma}  \ar[r]  & 0  \\
    0 \ar[r]   &      (K_{ n })^J  \ar[r]   & (F_{  n-1 } )^J  \ar[r]    &  (K_{  n-1 } )^J  \ar[r]  &
    0}$$
From Lemma \ref{lem-TOr-Ext} ($1\Leftrightarrow 2$), $\beta$ and
$\gamma$ are isomorphisms. Then, using snake lemma
\cite[Proposition 1.2.13]{Rel-hom}, we get that  $\alpha$ is also
an isomorphism. Then, by Lemma \ref{lem-TOr-Ext}
($2\Leftrightarrow 1$), $K_{ n }$ is finitely
presented and therefore $M$ is $n+1$-presented.\\
\indent It remains to prove the equivalence $1\Leftrightarrow 4$.
The implication $1\Rightarrow 4 $ follows from Theorem
\ref{thm-TOr-Ext} ($1\Leftrightarrow 4 $). Using Lemma
\ref{lem-TOr-Ext} ($4\Leftrightarrow 1$), the proof of the
implication   $4\Rightarrow 1$ is similar to the one  of the
implication $2\Rightarrow 1$ above.
\end{proof}

Now we give a counterpart of Cheatham and Stone's characterization
of $n$-$\mathscr{X}$-coherent rings using the notion of character
module. For that we need some results.

\begin{lem} \label{lem-pro-sum} Let
$\mathscr{X}$ be a class of $R$-modules such that, for a positive
integer $n\geq 1$, the subclass $\mathscr{X}_n$ of $n$-presented
$R$-modules of $\mathscr{X}$ is not empty. Then, for a family
$(M_{j})_{j\in J}$ of   $R$-modules, we have:
\begin{enumerate}
    \item $  \displaystyle\bigoplus_{j\in J} M_{j} $
       is $n$-$\mathscr{X}$-flat if and only if each $M_j$ is
       $n$-$\mathscr{X}$-flat.
 \item $  \displaystyle\prod_{j\in J} M_{j} $
       is $n$-$\mathscr{X}$-injective if and only if each $M_j$ is
       $n$-$\mathscr{X}$-injective.
           \end{enumerate}
\end{lem}
\begin{proof} $1.$ Follows from the isomorphism \cite[Theorem 8.10]{Rot}:
 $\Tor^R_n( N,  \displaystyle\bigoplus_{j\in J} M_{j}) \cong   \displaystyle\bigoplus_{j\in J}
Tor^R_n(N,M_{j}) .$\\
$2.$ Follows from the isomorphism \cite[Theorem 7.14]{Rot}:
 $\Ext^n_R( N, \displaystyle\prod_{j\in J} M_{j}) \cong
\displaystyle\prod_{j\in J} \Ext^n_R(N,M_{j}) .$
\end{proof}

We also need the following extension of the well-known Lambek's
result \cite{Lambek}:

\begin{lem} \label{lem-lambek} Let
$\mathscr{X}$ be a class of $R$-modules such that, for a positive
integer $n\geq 1$, the subclass $\mathscr{X}_n$ of $n$-presented
$R$-modules of $\mathscr{X}$ is not empty. Then,  a left (resp.
right) $R$-module $M$ is $n$-$\mathscr{X}$-flat if and only if
$M^{* }$ is an $n$-$\mathscr{X}$-injective right (resp. left)
$R$-module.
\end{lem}
\begin{proof}  Follows from the isomorphism \cite[page 360]{Rot}:
 $(\Tor^R_n( M, N))^* \cong  \Ext^n_R(N,  M ^* ).$ 
\end{proof}

The notion of pure submodules is also used. Recall that a short
exact sequence of    $R$-modules  $0\rightarrow A \rightarrow
B\rightarrow C \rightarrow 0$ is said to be pure if, for every
right $R$-module $M$, the sequence $0\rightarrow M \otimes_R
A\rightarrow M \otimes_R B\rightarrow M \otimes_R C \rightarrow 0$
is exact. In this case, $A$ is called a pure submodule of $B$.

\begin{lem}[\cite{Stenst2}, Exercise  40]   \label{lem-pur-split} Let    $0\rightarrow A \rightarrow B\rightarrow C
\rightarrow 0$  be a short exact sequence of    $R$-modules. Then,
the following assertions are equivalent:
\begin{enumerate}
    \item  The exact sequence  $0\rightarrow A \rightarrow B\rightarrow C
\rightarrow 0$ is pure;
     \item The exact sequence  $0\rightarrow \Hom_R(P,A) \rightarrow \Hom_R(P,B)\rightarrow \Hom_R(P,C)
\rightarrow 0$ is exact for every finitely presented $R$-module
$P$;
    \item The short sequence of right $R$-modules
$0\rightarrow C^* \rightarrow B^*\rightarrow A^* \rightarrow 0$
splits.
     \end{enumerate}
\end{lem}

\begin{lem}[\cite{Stenst2}, Exercise 41]   \label{lem-pur-mono}  Every $R$-module $M$ is a pure submodule of $M^{**}$ via the
the canonical monomorphism $M\rightarrow M^{**}$.
\end{lem}

\begin{lem}[\cite{CS}, Lemma 1]   \label{lem-pur-sum-prod}  For every family $(P_{j})_{j\in J}$ of
 left or right $R$-modules, we have:
\begin{enumerate}
    \item The sum $\displaystyle\bigoplus_{j\in J} P_{j}
 $ is a pure submodule of the product $\displaystyle\prod_{j\in J} P_{j}$
    \item  If each $P_i$ is a pure submodule of an $R$-module $Q_i$,
    then $\displaystyle\prod_{j\in J} P_{j}
 $ is a pure submodule of $\displaystyle\prod_{j\in J} Q_{j} $.
\end{enumerate}
\end{lem}

The following result is well-known for the classical flat case
(see, for instance, \cite[Theorem 11.1 $(a  \Leftrightarrow
 c)$]{Stenst2}).

\begin{lem}  \label{lem-pur-fla-inj}  Let
$\mathscr{X}$ be a class of $R$-modules such that, for a positive
integer $n\geq 1$, the subclass $\mathscr{X}_n$ of $n$-presented
$R$-modules of $\mathscr{X}$ is not empty.
\begin{enumerate}
    \item  Every pure submodule of an $n$-$\mathscr{X}$-flat
$R$-module is $n$-$\mathscr{X}$-flat.
    \item Every pure submodule of an $n$-$\mathscr{X}$-injective
$R$-module is $n$-$\mathscr{X}$-injective.
\end{enumerate}
\end{lem}
\begin{proof} $1 $. Let $A$ be a pure submodule of an
$n$-$\mathscr{X}$-flat $R$-module $B$. Then, by Lemma
\ref{lem-pur-split}  $(1 \Leftrightarrow 3)$, the sequence
$0\rightarrow (B/A)^* \rightarrow B^*\rightarrow A^* \rightarrow
0$ splits. Then, by Lemma \ref{lem-pro-sum}(2) and being a direct
summand of the $n$-$\mathscr{X}$-injective right $R$-module $B^*$
(Lemma \ref{lem-lambek}), the right $R$-module $A^*$ is
$n$-$\mathscr{X}$-injective. Therefore, by Lemma \ref{lem-lambek},
$A$ is $n$-$\mathscr{X}$-flat.\\
$2$.  Let $A$ be a pure submodule of an
$n$-$\mathscr{X}$-injective $R$-module $B$.  Consider an
$R$-module $M\in \mathscr{X}_n$. Then, there is an exact sequence
of $R$-modules: $$ F_{ n }\rightarrow F_{ n-1}\rightarrow \cdots
\rightarrow F_{ 1} \rightarrow F_0\rightarrow M \rightarrow 0$$
such that each $F_i$ is  finitely generated and free. Consider $
K_{ n }= \Ima( F_{ n }\rightarrow F_{ n-1})$ and $ K_{ n-1 }=
\Ima( F_{ n-1 }\rightarrow F_{ n-2})$. Then, we have the short
exact sequence $$ 0 \rightarrow  K_{ n } \rightarrow F_{
n-1}\rightarrow K_{ n-1}\rightarrow 0 $$ Since
$\Ext^1_R(K_{n-1},A)\cong Ext^n_R(M,A)$, we have only to prove
that $\Ext^1_R(K_{n-1},A)=0$. Applying the functor  $\Hom_R(K_{
n-1},-)$ to the short exact sequence $0\rightarrow A \rightarrow
B\rightarrow B/A \rightarrow 0$, we get the following  exact
sequence:
$$  \Hom_R(K_{
n-1},B)\rightarrow \Hom_R(K_{ n-1},B/A) \rightarrow
\Ext^1_R(K_{n-1},A)  \rightarrow \Ext^1_R(K_{n-1},B)$$  Since $A$
is  $n$-$\mathscr{X}$-injective,  $\Ext^1_R(K_{n-1},B)\cong
\Ext^n_R(M,A)=0$. Thus, the exact sequence above becomes
$$ (\alpha)\qquad \Hom_R(K_{ n-1},B)\rightarrow \Hom_R(K_{
n-1},B/A) \rightarrow \Ext^1_R(K_{n-1},A)  \rightarrow 0$$ On the
other hand,  since $M$ is $n$-presented, $K_{ n-1}$ is finitely
presented, and so, by Lemma \ref{lem-pur-split} $(1
\Leftrightarrow 2)$, we have the following exact sequence:
$$ (\beta)\qquad \Hom_R(K_{ n-1},B)\rightarrow \Hom_R(K_{
n-1},B/A) \rightarrow 0$$  Therefore, by the sequences $(\alpha) $
and $(\beta)$ above, we get $\Ext^1_R(K_{n-1},A) =0$.
\end{proof}

Now we are ready to prove our second main result.

\begin{thm}\label{thm-character} Let
$\mathscr{X}$ be a class of $R$-modules such that, for a positive
integer $n\geq 1$, the subclass $\mathscr{X}_n$ of $n$-presented
$R$-modules of $\mathscr{X}$ is not empty.  Then, the following
assertions are equivalent:
\begin{enumerate}
    \item   $R$ is left $n$-$\mathscr{X}$-coherent;
     \item  An $R$-module  $M$ is  $n$-$\mathscr{X}$-injective if and
     only if $M^*$ is $n$-$\mathscr{X}$-flat;
 \item   An $R$-module  $M$ is  $n$-$\mathscr{X}$-injective if and
     only if $M^{**}$ is $n$-$\mathscr{X}$-injective;
    \item A right $R$-module  $M$ is  $n$-$\mathscr{X}$-flat if and
     only if $M^{**}$ is $n$-$\mathscr{X}$-flat.
     \end{enumerate}
\end{thm}
\begin{proof} $1\Rightarrow 2$. Since $R$ is left
$n$-$\mathscr{X}$-coherent, every $n$-presented module in
$\mathscr{X}$ is infinitely presented, and so, from \cite[Theorem
9.51 and the remark following it]{Rot}, we have:
$$\Tor^R_n( M^*, N) \cong    (\Ext^n_R(N,M))^* $$
for every  $R$-module $N\in \mathscr{X}_n$. This shows that an
$R$-module $M$ is  $n$-$\mathscr{X}$-injective if and
     only if $M^*$ is $n$-$\mathscr{X}$-flat.\\
$2\Rightarrow 3$. Follows from the equivalence of $(2)$ and
Lemma \ref{lem-lambek}.\\
$3\Rightarrow 4$. Let $M$ be   $n$-$\mathscr{X}$-flat right
$R$-module. Then, by Lemma \ref{lem-lambek}, $M^*$ is
$n$-$\mathscr{X}$-injective, and by $(3)$, $M^{***}$ is
$n$-$\mathscr{X}$-injective. Therefore, by Lemma \ref{lem-lambek},
$M^{**}$  $n$-$\mathscr{X}$-flat.\\
Conversely, consider a right $R$-module $M$ such that $M^{**}$ is
$n$-$\mathscr{X}$-flat. By Lemma \ref{lem-pur-mono}, $M$ is a pure
submodule of $M^{**}$. Then,   Lemma \ref{lem-pur-fla-inj} $(1)$
shows that $M$ is $n$-$\mathscr{X}$-flat, too.\\
$4\Rightarrow 1$. Using Theorem \ref{thm-flatr-inj}, we have to
prove that every product of  $n$-$\mathscr{X}$-flat right
$R$-modules is $n$-$\mathscr{X}$-flat. Then, consider a family
$(P_{j})_{j\in J}$ of  $n$-$\mathscr{X}$-flat right $R$-modules.
From Lemma \ref{lem-pro-sum} $(1)$, the sum
$\displaystyle\bigoplus_{j\in J} P_{j} $ is
$n$-$\mathscr{X}$-flat. Then, by $(4)$,
$(\displaystyle\bigoplus_{j\in J} P_{j})^{**}\cong
(\displaystyle\prod_{j\in J} P_{j}^*)^*$ is
$n$-$\mathscr{X}$-flat. On the other hand, from Lemma
\ref{lem-pur-sum-prod} $(1)$, the sum $
\displaystyle\bigoplus_{j\in J} P_{j} ^{* }$ is a pure submodule
of the product $ \displaystyle\prod_{j\in J} P_{j} ^{* }$. Then,
by Lemma \ref{lem-pur-split} $(1 \Leftrightarrow 3)$, we deduce
that $( \displaystyle\bigoplus_{j\in J} P_{j} ^{* })^*$ is a
direct summand of $( \displaystyle\prod_{j\in J} P_{j} ^{* })^*$,
and so $ \displaystyle\prod_{j\in J} P_{j}^{**} \cong (
\displaystyle\bigoplus_{j\in J} P_{j} ^{* })^*$ is
$n$-$\mathscr{X}$-flat. Therefore, using  Lemmas
\ref{lem-pur-sum-prod} $(2)$ and \ref{lem-pur-fla-inj} $(1)$, the
direct product $ \displaystyle\prod_{j\in J} P_{j} $ is
$n$-$\mathscr{X}$-flat.
\end{proof}

We end this paper with a counterpart of \cite[Proposition
6.5.1]{Rel-hom} such that we give a characterization of
$n$-$\mathscr{X}$-coherent by $n$-$\mathscr{X}$-flat preenvelope.
Here, the $n$-$\mathscr{X}$-flat preenvelopes are Enochs'
$\mathscr{F}$-preenvelopes, where $\mathscr{F}$ is the class of
all $n$-$\mathscr{X}$-flat modules (see Introduction). The proof
of this result is analogous to the one of \cite[Proposition
6.5.1]{Rel-hom}. So we need the following two results:

\begin{lem}[\cite{Rel-hom}, Lemma 5.3.12]\label{lem-prenv1} Let $F$ and $N$ be $R$-modules.
Then, there is a cardinal number $\aleph_\alpha$ such that, for
every homomorphism $f\,:\,  N\rightarrow F$, there is a pure
submodule $P$ of $F$ such that $f(N)\subset P$ and $\Card(P)\leq
\aleph_\alpha$.
\end{lem}

\begin{lem}[\cite{Rel-hom}, Corollary 6.2.2] \label{lem-prenv2}
Let $\mathscr{X}$ be a class of $R$-modules that is closed under
direct products.   Let $M$ be an $R$-module with $\Card(M)=
\aleph_\beta$. Suppose that there is a cardinal $\aleph_\alpha$
such that, for an  $R$-module $F\in \mathscr{X}$   and a submodule
$N$ of $F$ with  $\Card(P)\leq \aleph_\beta$, there is a submodule
$P$ of $F$ containing $N$ with $P\in \mathscr{X}$ and
$\Card(P)\leq \aleph_\alpha$. Then, $M$ has an
$\mathscr{X}$-preenvelope.
\end{lem}

\begin{thm}\label{thm-prenv} Let
$\mathscr{X}$ be a class of $R$-modules such that, for a positive
integer $n\geq 1$, the subclass $\mathscr{X}_n$ of $n$-presented
$R$-modules of $\mathscr{X}$ is not empty.  Then,   $R$ is left
$n$-$\mathscr{X}$-coherent if and only if every right $R$-module
$M$ has an $n$-$\mathscr{X}$-flat preenvelope.
\end{thm}
\begin{proof} $\Rightarrow.$ Let $M$ be a right $R$-module with $\Card(M)=
\aleph_\beta$. From Lemma \ref{lem-prenv1}, there is a cardinal
$\aleph_\alpha$ such that, for an $n$-$\mathscr{X}$-flat right
$R$-module $F $ and a submodule $N$ of $F$ with  $\Card(P)\leq
\aleph_\beta$, there is a pure submodule $P$ of $F$  containing
$N$ and $\Card(P)\leq \aleph_\alpha$. From Lemma
\ref{lem-pur-fla-inj} $(1)$, $P$  is $n$-$\mathscr{X}$-flat.
Therefore, since the class of all $n$-$\mathscr{X}$-flat  right
$R$-modules is closed under direct products (by Theorem
\ref{thm-flatr-inj}), Lemma \ref{lem-prenv2} shows that $M$ has an
$n$-$\mathscr{X}$-flat preenvelope.\\
$\Leftarrow.$ To prove that $R$ is left
$n$-$\mathscr{X}$-coherent, it is sufficient, by Theorem
\ref{thm-flatr-inj}, to prove that every product of
$n$-$\mathscr{X}$-flat right $R$-modules is
$n$-$\mathscr{X}$-flat. Consider a family $(P_{j})_{j\in J}$ of
$n$-$\mathscr{X}$-flat right $R$-modules. By hypothesis,
$\displaystyle\prod_{j\in J} P_{j}$ has an $n$-$\mathscr{X}$-flat
preenvelope $f\,:\,\displaystyle\prod_{j\in J} P_{j} \rightarrow
F$. Then, for each canonical projection
$p_i\,:\,\displaystyle\prod_{j\in J} P_{j} \rightarrow P_i$ with
$i\in J$, there exists a homomorphism $p_i\,:\, F \rightarrow P_i$
such that $h_if=p_i$. Now, consider the homomorphism
$h=(h_j)_{j\in J}\,:\, F \rightarrow \displaystyle\prod_{j\in J}
P_{j}$ defined by $h(x)=(h_j(x))_{j\in J}$ for every $x\in F$.
Then, for every $a=(a_j)_{j\in J} \in \displaystyle\prod_{j\in J}
P_{j}$, we have:
$$ hf(a)=(h_j( f(a)))_{j\in J}=(p_j (a) )_{j\in J}=a.$$
This means that $ hf=1_{\Pi P_j}$. Then, $\displaystyle\prod_{j\in
J} P_{j}$ is a direct summand of $F$. Therefore, by Lemma
\ref{lem-pro-sum}, $\displaystyle\prod_{j\in J} P_{j}$ is
$n$-$\mathscr{X}$-flat.
\end{proof}

%\noindent {\bf Acknowledgment.}
%The authors would like to thank the referee for the valuable suggestions and comments.

%%%%%%%%%%%%%%%%%%%%%%%%%%%%%%%%%%%%%%%%%%%%%%%%%%%%%%%%%%%%
%%%%%%%%%%%%%%%%%%%%%%%%%%%%%%%%%%%%%%%%%%%%%%%%%%%%%%%%%
%%%%%%%%%%%%%%%%%%%%%%%%%%%%%%%%%%%%%%%%%%%%%%%%%%%%%%%%%
%%%REFERENCES%%%%%%%%%%%%%%%%%%%%%%%%%%%%%%%%%%%%%%%%%%%%
%%%%%%%%%%%%%%%%%%%%%%%%%%%%%%%%%%%%%%%%%%%%%%%%%%%%%%%%

\bigskip\bigskip

%%%%%%%%%%%%%%%%%%%%%%%%%%%%%%%%%%%%%%%%%%%%%%%%%%%%%%%%
\end{document}